\documentclass[12pt]{article}
\usepackage{geometry,amsmath,amssymb,bbm,theorem,amscd,graphicx}
\usepackage{colortbl}
\usepackage{ytableau}
\usepackage{tikz}
\usepackage{comment}

\newenvironment{proof}{\noindent\textbf{Proof\ }}{\hspace*{\fill}$\Box$\medskip}

{\theorembodyfont{\rmfamily}}
\newtheorem{Definition}{Definition}
\newtheorem{Lemma}{Lemma}
\newtheorem{Proposition}{Proposition}

\newtheorem{Theorem}{Theorem}

\newenvironment{MSC}{%
\smallbreak
\noindent \textbf{2010\ Mathematics Subject Classification\,:}}

\newenvironment{keywords}{%
\noindent\textbf{Key words and phrases\,:}\itshape}

\usepackage[english]{babel} %

\begin{document}


\title{Schur type poly-Bernoulli numbers}
\author{Naoki Nakamura and Maki Nakasuji}
\date{}
\maketitle

\begin{abstract} 
The poly-Bernoulli numbers and its relative are defined by the generating series using the polylogarithm series, and we call them type $B$ and $C$, respectively.
As a generalization of these poly-Bernoulli numbers, we introduce Schur type poly-Bernoulli numbers and investigate their properties.
First, we define a generalization of Arakawa-Kaneko multiple zeta functions and obtain their expression in terms of   Schur type Bernoulli numbers. 
Next, under the restriction to the hook type, we define a generalization of Kaneko-Tsumura multiple zeta functions and obtain similar expression in terms of 
Schur type Bernoulli numbers. Lastly, we study more properties such as a recurrence formula, a relation formula between Bernoulli numbers and a description in terms of the Stirling numbers. 
\begin{MSC}
 11M41, 
 05E05. 
\end{MSC} 
\end{abstract}

%
%
%
%

\section{Introduction}
The poly-Bernoulli numbers ${ \mathbb B}_n^{(k)}$ for $n\geq 0 , k\in {\mathbb Z}$, which are a natural generalization of the classical Bernoulli numbers, were first introduced by Kaneko (\cite{K}).
These numbers and the related numbers ${ \mathbb C}_n^{(k)}$ are defined by the generating series using the polylogarithm series $Li_k(t)$ as follows:
$$\frac{Li_k(1-e^{-t})}{1-e^{-t}}=\sum_{n=0}^{ \infty}{ \mathbb B}_n^{(k)}\frac{t^n}{n!}, \quad
\frac{Li_k(1-e^{-t})}{e^{t}-1}=\sum_{n=0}^{ \infty}{ \mathbb C}_n^{(k)}\frac{t^n}{n!},
$$
where the polylogarithm series $Li_k(t)$ is the formal power series $\displaystyle{\sum_{n=1}^{\infty}\frac{t^n}{n^k}}$ for any integer $k$.
For $t=1$, the polylogarithm series reduces to the special value of the Riemann zeta function, that is, $Li_k(1)=\zeta(k)$.
Arakawa and Kaneko \cite{AKbook}  introduced the multi-logarithmic function  ${Li}_{k_1, k_2, \cdots k_r} (z)$ defined below \eqref{polylog}, and for the purpose of establishing a connection between the multiple zeta values and the poly-Bernoulli numbers, they gave an integral representation and the analytic continuation of single-variable multiple zeta functions, and
investigated the properties of ${ \mathbb B}_n^{(k)}$, such as the closed formulas in terms of the Stirling numbers of the second kind and the 
relation between ${ \mathbb B}_n^{(k)}$ and ${ \mathbb C}_n^{(k)}$
in \cite{AKbook}. 
They further introduced a new kind of multiple zeta functions and a generalization for multiple indices called Arakawa-Kaneko multiple zeta functions:
$$\xi_k(s)=\frac{1}{\Gamma(s)}\int_0^{ \infty}\frac{Li_k(1-e^{-t})}{e^t-1}t^{s-1}dt,
$$
\begin{equation}\label{AKMZ}
\xi_{k_1, \cdots, k_r}(s)=\frac{1}{\Gamma(s)}\int_0^{ \infty}\frac{Li_{k_1, \cdots, k_r}(1-e^{-t})}{e^t-1}t^{s-1}dt,
\end{equation}
where $k$ and $k_1, \cdots, k_r \geq 1$ are integers and
\begin{eqnarray}\label{polylog} 
 Li_{k_1, k_2, \cdots k_r} (z) = \sum_{0 < m_1 < m_2 <\cdots < m_r} \frac{z^{m_r}}{m_1^{k_1} m_2^{k_2} \cdots m_r^{k_r} }.
\end{eqnarray}
They showed that $\xi_k(s)$ and $\xi_{k_1, \cdots, k_r}(s)$ can be analytically continued to entire functions of $s\in {\mathbb C}$, and the values at nonpositive integers interpolate  ${\mathbb C}_n^{(k)}$.
Kaneko and Tsumura (\cite{KT}) defined and studied a twin sibling of the Arakawa-Kaneko multiple zeta functions:
\begin{equation}\label{KTMZ}
\eta_{k_1, \cdots, k_r}(s)=\frac{1}{\Gamma(s)}\int_0^{ \infty}\frac{Li_{k_1, \cdots, k_r}(1-e^{t})}{1-e^t}t^{s-1}dt.
\end{equation}
This function also can be analytically continued to an entire functions on the whole complex plane and the values at nonpositive integers interpolate
${\mathbb B}_n^{(k)}$.
These functions have since been studied by many authors (e.g., \cite{CC}, \cite{Y}).

Nakasuji, Phuksuwan and Yamasaki \cite{NPY} defined the Schur multiple zeta functions $\zeta_{ \lambda}(\bf s)$, which are a generalization of the Euler-Zagier type multiple zeta and zeta-star functions.
These functions are defined as sums over combinatorial objects called semi-standard Young tableaux of shape $\lambda$ (see Definition \ref{defschur}).
In the present article, we introduce Schur type polylogarithm series $Li_{\bf k}^{ \lambda}({\bf z})$ for $\mathbf{z} = (z_1, z_2, \cdots z_c)$ such that $| z_i | < 1$ $(1 \leq i \leq c)$ , which reduce to the Schur multiple zeta values when ${\bf z} \to {\bf 1}$: 
\begin{equation}\label{schurpolylog}
Li_{{\mathbf{k}}}^\lambda ({\bf z}) := \sum_{(m_{ij)} \in SSYT(\lambda)} \frac{{ \bf z}^{m_c}}{\prod_{(i,j)\in D(\lambda)}m_{ij}^{k_{ij}}},
 \end{equation}
where ${\bf k}\in T(\lambda, {\mathbb N})$ and $\displaystyle{\mathbf{z}^{m_c} = \prod_{(i, j) \in C(\lambda) } z_{ij}^{m_{ij}}}$.
For an explanation of the notation $SSYT( \lambda)$, $D(\lambda)$, $T(\lambda, {\mathbb N})$ and  $C( \lambda)$,
refer to the next section. And we define Schur type Bernoulli numbers using $Li_{\bf k}^\lambda (\mathbf{z})$ as an analogue of \cite{AKbook} as follows :

\begin{equation}\label{defBernoulli}
\frac{Li_\mathbf{k}^{\lambda}(1-e^{-z_1}, 1-e^{-z_2}, \cdots, 1-e^{-z_c})}{(1-e^{-z_1})(1-e^{-z_2})\cdots (1-e^{-z_c})}
=\sum_{m_1, \cdots, m_c=0}^{\infty}{\mathbb B}_{m_1, \cdots, m_c}^{\lambda, \mathbf{k}}\frac{z_1^{m_1}\cdots z_c^{m_c}}{m_1!\cdots m_c!},
\end{equation}
\begin{equation}\label{defBernoulliC}
\frac{Li_\mathbf{k}^{\lambda}(1-e^{-z_1}, 1-e^{-z_2}, \cdots, 1-e^{-z_c})}{(e^{z_1}-1)(e^{z_2}-1)\cdots (e^{z_c}-1)}
=\sum_{m_1, \cdots, m_c=0}^{\infty}{\mathbb C}_{m_1, \cdots, m_c}^{\lambda, \mathbf{k}}\frac{z_1^{m_1}\cdots z_c^{m_c}}{m_1!\cdots m_c!}.
\end{equation}
where ${\bf z}=(z_1, z_2, \cdots, z_c)$, $c=|C(\lambda)|$ and ${\bf k}\in T(\lambda, { \mathbb N})$.

Furthermore, we define a generalization of the Arakawa-Kaneko multiple zeta functions to Schur type functions and will obtain the expression the generalizations using Schur type Bernoulli numbers (Theorem \ref{xi}).
If we restrict them to the hook type $\lambda = (h, 1^{\ell - 1})$, then we can define a generalization of Kaneko-Tsumura multiple zeta functions to Schur type functions and will obtain the expression in terms of Schur type Bernoulli numbers (Theorem \ref{eta}). Further we will obtain a generalization of the known results for the properties of Bernoulli numbers such as 
a recurrence formula, a relation formula between Bernoulli numbers ${\mathbb B}$ and ${\mathbb C}$, and a description in terms of the Stirling numbers (Section 4).

\section{Schur multiple zeta functions}
 A  partition 
$ \lambda = ( \lambda_1 , \ldots , \lambda_r) $ of a positive integer $n$ is a non-increasing sequence of positive integers such that 
$$
| \lambda | = \sum^{r}_{i=1} \lambda_i = n .
$$
We call ${ | \lambda | }$ the weight of $\lambda$ and $ l(\lambda) = r$ the length of $\lambda$, and write $ \lambda \vdash n $. And we sometimes express it as $ \lambda = ( n^{m_n(\lambda)}, \ldots , 2^{m_2(\lambda)} , 1^{m_1(\lambda)}) $, where $m_i(\lambda)$ is the multiplicity of $i$ in $\lambda$.
We arrange the $\lambda_i$ boxes left-aligned in the $i$th row from top to bottom. 
For example, 
for $\lambda = (3,1)$,
$
\ytableausetup{boxsize=normal}  
\begin{ytableau}
 \ &   & \\
  \\
\end{ytableau}
$.
We identify $ \lambda \vdash n $ with the Young diagram $ D(\lambda) = \{ ( i , j ) \in {\mathbb Z}^2 | 1 \leqq i \leqq r , 1 \leqq j \leqq \lambda_1 \} $.
If there are $(i+1 , j) \notin D(\lambda) $ and $(i, j+1) \notin D(\lambda) $, we say that  $(i, j) \in D(\lambda) $ is a corner, and define $C(\lambda) \in D(\lambda)$ as the set of all corners of $\lambda$.
For example, 
for $\lambda = (4,3,3,1)$,
$
C((4,4,3,1)) = \{ (2,4), (3,3), (4,1) \}
$.
Let $X$ be a set. For a partition $\lambda$, $T = (t_{ij}) $ which is a $D(\lambda)$ putting $t_{ij} \in X$ into a box of $D(\lambda)$ is a Young tableau of shape $\lambda $ over $X$. 
Let $T(\lambda, X)$ be the set of all Young tableaux of shape $\lambda$ over $X$. We identify $T(\lambda, X)$ with $X^{|\lambda|}$. 
We call $T = (t_{ij})$ a semi-standard Young tableau, where we put positive integers $t_{ij}$  into the boxes such that they are weakly increasing in each row from left to right and are strictly increasing in each column from top to bottom. We denote the set of all semi-standard Young tableaux of shape $\lambda$ by SSYT($\lambda$) .
Using this notation, Nakasuji, Phuksuwan and Yamasaki \cite{NPY} introduced  Schur multiple zeta functions as follows.

\begin{Definition}\label{defschur}
For ${\bf s} = (s_{ij}) \in T(\lambda, \mathbb{C})$, we define
\begin{eqnarray} \label{MSZFs}
\zeta_\lambda({\bf s}) = \sum_{M \in SSYT(\lambda) } \frac{1}{M^{\bf s}},
\end{eqnarray}
where $\displaystyle{M^{\bf s} = \prod _{(i , j) \in D(\lambda)} m_{ij}^{s_{ij}} , M = (m_{ij}) \in \rm{SSYT}(\lambda)}$. Here, we set $\zeta_\lambda=1 $ for $\lambda = \phi $. We call $\zeta_\lambda({\bf s})$ a Schur multiple zeta function. 
\end{Definition}

\begin{Lemma}[ {\cite[Lemma 2.1]{NPY} }]  \label{abscon}
We define
\begin{eqnarray}
  W_\lambda = \left \{ 
  \mathbf{s} = (s_{ij}) \in T(\lambda, \mathbb{C})
 \middle | \begin{array}{l}
  {\rm Re}(s_{ij}) \in \mathbb{R}_{\geq 1}, (i,j) \in D(\lambda) \backslash C(\lambda) \\
   {\rm Re}(s_{ij}) \in \mathbb{R}_{ > 1}, (i,j) \in  C(\lambda) 
  \end{array} \right \}.
\end{eqnarray}
Then, $(\ref{MSZFs})$ converges absolutely for $\mathbf{s} \in W_\lambda$ .
\end{Lemma}

Recall  that  Schur multiple zeta functions can be written as a linear combination of multiple zeta functions or multiple zeta-star functions. 

\begin{Definition}
For $\lambda $, $\it{F}$  is the set of all bijections $f: D(\lambda) \rightarrow  \{ 1,2, \cdots ,n \} $ that satisfying the following two conditions : \\
$\rm( \hspace{.18em} i \hspace{.18em})$  for all $i$, $f((i,j)) < f((i,j')) $ if and only if $j < j'$, \\
$\rm(\hspace{.08em}ii\hspace{.08em})$ for all $j$, $f((i,j)) < f((i',j)) $ if and only if $i < i'$.\\
For $T = (t_{ij}) \in T(\lambda,X)$,  we define 
\begin{eqnarray*}
V(T) = \{ ( t_{f^{-1}(1)}, t_{f^{-1}(2)}, \cdots , t_{f^{-1}(n)} ) \}.
\end{eqnarray*}
\end{Definition}


\begin{Definition} \label{in}
When X  has an addition +, we write $ { \bf w} \preceq T  $ for ${ \bf w} = (w_1, \ldots , w_n) \in X^m$ if there exists $(v_1, \cdots ,v_n ) \in V(T)$ satisfying the following:  \\
$\rm( \hspace{.18em} i \hspace{.18em})$ $w_k = v_{h_k} + v_{h_k+1} + \cdots + v_{h_k + l_k}$, \\
$\rm(\hspace{.08em}ii\hspace{.08em})$ for any $j$, there exist no $i$ and $i'$ such that $t_{ij}, t_{i'j} \in \{ v_{h_k} +, v_{h_k+1} , \cdots , v_{h_k + l_k} \}  $ for $i \neq i' $, \\
$\rm(\hspace{.08em}iii\hspace{.08em})$ $\bigsqcup_{k=1}^{m} \{ h_k, h_{k}+1, \ldots h_k + l_k \} = \{ 1,2, \ldots , n \}$. 
\end{Definition}
Then we have  
\begin{eqnarray}
\zeta_\lambda({\bf s}) &=& \sum_{{\bf t} \preceq {\bf s} } \zeta ({\bf t}),   \label{zetalambdas} \\
\zeta_\lambda({\bf s}) &=& \sum_{{\bf t} \preceq {\bf s}' } (-1)^{n-l({\bf t})} \zeta^{\star} ({\bf t}).  \label{zetalambdass}
\end{eqnarray} 
Here, $\zeta({\bf t})$ and $\zeta^{ \star}({\bf t})$ are original multiple zeta and multiple zeta-star functions, respectively:
$$
\zeta(s_1, s_2, \cdots, s_r)=\sum_{1\leq m_1<m_2<\cdots < m_r}\frac{1}{m_1^{s_1}m_2^{s_2}\cdots m_r^{s_r}},
$$
and
$$
\zeta^{\star}(s_1, s_2, \cdots, s_r)=\sum_{1\leq m_1\leq m_2\leq \cdots \leq m_r}\frac{1}{m_1^{s_1}m_2^{s_2}\cdots m_r^{s_r}}.
$$
And, ${\bf s}' =(s_{ij}')\in W_{\lambda'}$, where $\lambda'$ is the conjugate of $\lambda$.

The multiple zeta functions have an integral expression, as shown in the following proposition:
\begin{Proposition}{\rm(\cite[Proposition 2 (i)]{AKbook}, \cite[Theorem3 (i)]{AKbook})}\label{AKzg}
Let $k_1, \cdots, k_{r-1}\geq 1$ be integers.
For ${\rm Re}(s) > 1$, 
\begin{eqnarray}
 \zeta(k_1, k_2, \cdots k_{r-1} , s ) &=& \frac{1}{\Gamma(s)} \int_0^{\infty} \frac{t^{s-1}}{e^t-1} Li_{k_1, k_2, \cdots k_{r-1}} (e^{-t}) dt \\
&=&  \frac{1}{ \Gamma(k_{1}) \cdots \Gamma(k_{r-1})  \Gamma(k_{s})} \int_0^{\infty} \cdots \int_0^{\infty} x_{1}^{{k_{1}}-1} \cdots x_{r-1}^{{k_{r-1}}-1} x_{r}^{s-1} \notag \\
&& \frac{1}{ (e^{(x_{1}+ \cdots +x_{r})} - 1)(e^{(x_{2}+ \cdots +x_{r})} - 1) \cdots (e^{x_{r}} - 1)} dx_{1}  \cdots dx_{r-1} dx_{r} .  \notag
\end{eqnarray}
\end{Proposition}
A similar calculation gives the following proposition.
\begin{Proposition}
Let $k_1, \cdots, k_{r-1}\geq 1$ be integers.
For ${\rm Re}(s) > 1$, 
\begin{eqnarray}
 \zeta^{\star}(k_1, k_2, \cdots k_{r-1} , s ) &=& \frac{1}{\Gamma(s)} \int_0^{\infty} \frac{e^tt^{s-1}}{e^t-1} Li_{k_1, k_2, \cdots k_{r-1}} (e^{-t}) dt. \\
&=&  \frac{1}{ \Gamma(k_{1}) \cdots \Gamma(k_{r-1})  \Gamma(k_{s})} \int_0^{\infty} \cdots \int_0^{\infty} x_{1}^{{k_{1}}-1} \cdots x_{r-1}^{{k_{r-1}}-1} x_r^{s-1} \notag \\
&& \frac{e^{x_1+2x_2+\cdots + r x_r}}{ (e^{(x_{1}+x_{2}+ \cdots  x_r)} - 1)(e^{(x_{2}+ \cdots x_r)} - 1) \cdots (e^{x_r} - 1)} dx_{1}  \cdots dx_{r-1} dx_r . \notag
\end{eqnarray}
\end{Proposition}

Schur multiple zeta functions for $\lambda = (h, 1^{\ell-1})$ also have an integral expression.
\begin{Proposition} \label{znlG}
For $\lambda = (h,1^{\ell-1})$, 
\begin{equation}\label{zetagamma}
\zeta_\lambda (\mathbf{k}) = \sum_{{\bf j} \preceq {\bf k} } \zeta ({\bf j}) =  \sum_{{\bf j} \preceq {\bf k} } \frac{1}{\Gamma(s)} \int_0^{\infty} \frac{t^{s-1}}{e^t-1} Li_{{\bf j^{-}}} (e^{-t}) dt,
\end{equation}
where $ {\bf j^{-}} = (j_1, j_2, \cdots j_{h-1} ) $ for $\mathbf{j} = (j_1, j_2, \cdots j_{h-1} , s ) $, and further
\begin{align}
 \zeta_\lambda({\bf k})
= & \frac{1}{ \Gamma(k_{11}) \cdots \Gamma(k_{1h})  \Gamma(k_{21}) \cdots \Gamma(k_{\ell 1})} \int_0^{\infty} \cdots \int_0^{\infty} x_{11}^{{k_{11}}-1} \cdots x_{\ell 1}^{{k_{\ell 1}}-1} x_{12}^{k_{12}-1} \cdots x_{1h}^{k_{1h}}  \notag\\ 
  & \frac{1}{e^{x_{11} + \cdots + x_{\ell 1} + x_{12} \cdots + x_{1h}} -1}  \frac{1}{e^{( x_{21} + \cdots + x_{\ell 1})}-1}  \cdots \frac{1}{e^{ x_{\ell 1}}-1}  \frac{e^{(x_{12}+ \cdots + x_{1h})}}{ e^{(x_{12}+ \cdots + x_{1h})} - 1}\cdots \frac{e^{x_{1h}}}{e^{x_{1h}} - 1} \notag\\
  & dx_{11} \cdots dx_{\ell 1} dx_{12} \cdots dx_{1h}.\label{integralexpression}
\end{align}
\end{Proposition}

 \begin{proof}
 The first assertion is obtained from $(\ref{zetalambdas} )$ and Proposition $\ref{AKzg}$.
For the second assertion, we can write
\begin{eqnarray*}
 \zeta_\lambda({\bf k})
  &= & \sum_{\substack{0 < m_{11} \leq \cdots \leq m_{1h}, \\ 0 < m_{11} < \cdots < m_{\ell1} }}  \frac{1}{m_{11}^{k_{11}} \cdots m_{1h}^{k_{1h}} m_{21}^{k_{21}} \cdots m_{\ell 1}^{k_{\ell 1}}} \\
 &= &\sum_{m_{11} = 1}^\infty  \frac{1}{m_{11}^{k_{11}} } \left (  \sum_{m_{11} \leq m_{12} \leq \cdots  \leq m_{1h}} \frac{1}{m_{12}^{k_{12}} \cdots  m_{1h}^{k_{1h}}}  \right) \left (  \sum_{m_{11} < m_{21} < \cdots < m_{\ell1}} \frac{1}{m_{21}^{k_{21}} \cdots  m_{\ell 1}^{k_{ \ell 1}}} \right ) \\
& = & \sum_{m_{11} = 1}^\infty  \frac{1}{m_{11}^{k_{11}} } \left (  \sum_{m_{11} \leq m_{12} \leq \cdots  \leq m_{1,h-1}} \frac{1}{m_{12}^{k_{12}} \cdots  m_{1,h-1}^{k_{1,h-1}}} \sum_{m_{1h} = m_{1,h-1}}^{ \infty} \frac{1}{m_{1h}^{k_{1h}}}   \right)  \\
&& \left (  \sum_{m_{11} < m_{21} < \cdots < m_{\ell-1,1}} \frac{1}{m_{21}^{k_{21}} \cdots  m_{\ell-1,1}^{k_{ \ell-1,1}}} \sum_{m_{\ell1} = m_{\ell-1,1} +1 }^{\infty} \frac{1}{m_{\ell 1}^{k_{\ell 1}}} \right ) .
\end{eqnarray*}

Using Proposition $\ref{AKzg}$, this equals
 \allowdisplaybreaks[1]
\begin{eqnarray*}
\hspace{5mm}& =&   \sum_{m_{11} = 1}^\infty  \frac{1}{m_{11}^{k_{11}} } \cdot\\
& &\left (  \sum_{m_{11} \leq m_{12} \leq \cdots  \leq m_{1,h-1}} \frac{1}{m_{12}^{k_{12}} \cdots  m_{1,h-1}^{k_{1,h-1}}} \sum_{m_{1h} = m_{1,h-1}}^{\infty} \frac{1}{\Gamma(k_{1h})} \int_0^\infty e^{-m_{1h} x_{1h}} x_{1h}^{k_{1h}-1} dx_{1h} \right)  \\
& &\left (  \sum_{m_{11} < m_{21} < \cdots < m_{\ell-1,1}} \frac{1}{m_{21}^{k_{21}} \cdots  m_{\ell-1,1}^{k_{ \ell-1,1}}} \sum_{m_{\ell1} = m_{\ell-1,1} +1 }^{\infty}  \frac{1}{\Gamma(k_{\ell1})} \int_0^\infty e^{-m_{\ell1} x_{\ell1}} x_{\ell1}^{k_{\ell1}-1} dx_{\ell1} \right ). 
\end{eqnarray*}
Setting $a = m_{1h} - m_{1,h-1} \text{ and } b = m_{\ell1} - m_{\ell-1,1}$, 
$$\sum_{m_{1h} = m_{1,h-1}}^{\infty}e^{-m_{1h} x_{1h}}=\sum_{a = 0}^\infty e^{-a x_{1h} }e^{-m_{1,h - 1} x_{1h}}=\frac{x_{1h}^{k_{1h}-1} e^{x_{1h}}}{e^{x_{1h}}-1}e^{-m_{1,h - 1} x_{1h}},$$
$$\sum_{m_{\ell1} = m_{\ell-1,1} +1 }^{\infty}e^{-m_{\ell1} x_{\ell1}}=\sum_{b =1}^\infty e^{-b x_{\ell1}}e^{-m_{\ell-1, 1}x_{\ell1}}=\frac{x_{\ell 1}^{k_{\ell1}-1}}{e^{ x_{\ell1}} - 1}e^{-m_{\ell-1, 1}x_{\ell1}},
$$
so
\begin{eqnarray*}
 \zeta_\lambda({\bf k})& =&   \sum_{m_{11} = 1}^\infty  \frac{1}{m_{11}^{k_{11}} } \left ( \frac{1}{\Gamma(k_{1h})}  \int_0^\infty  \frac{x_{1h}^{k_{1h}-1} e^{x_{1h}}}{e^{x_{1h}}-1}  \sum_{m_{11} \leq m_{12} \leq \cdots  \leq m_{1,h-1}} \frac{e^{-m_{1,h - 1} x_{1h}}}{m_{12}^{k_{12}} \cdots  m_{1,h-1}^{k_{1,h-1}}}  dx_{1h} \right)  \\
&& \left ( \frac{1}{\Gamma(k_{\ell1})}  \int_0^\infty    \frac{x_{\ell 1}^{k_{\ell1}-1}}{e^{ x_{\ell1}} - 1}   \sum_{m_{11} < m_{21} < \cdots < m_{\ell-1,1}} \frac{e^{-m_{\ell-1, 1}x_{\ell1}}}{m_{21}^{k_{21}} \cdots  m_{\ell-1,1}^{k_{ \ell-1,1}}} dx_{\ell1} \right ) .
\end{eqnarray*}
Applying the same argument repeatedly, we obtain
\begin{eqnarray*}
  &=&  \frac{1}{ \Gamma(k_{12}) \cdots \Gamma(k_{1h})  \Gamma(k_{21}) \cdots \Gamma(k_{\ell 1})} \int_0^{\infty} \frac{x_{1h}^{k_{1h}-1} e^{x_{1h}}}{e^{x_{1h}}-1} \int_0^\infty \cdots \int_0^\infty  \frac{x_{12}^{k_{12}-1}e^{(x_{12 }+ \cdots + x_{1h})}}{e^{(x_{12 }+ \cdots + x_{1h})}-1} \\
   && \int_0^{\infty}  \frac{x_{\ell 1}^{k_{\ell 1}-1}}{e^{x_{ \ell 1}}-1} \int_0^\infty  \cdots \int_0^\infty \frac{x_{21}^{k_{21}-1}}{e^{(x_{21 }+ \cdots + x_{\ell 1})}-1}  \\
   &&
 \int_0^\infty \sum_{m_{11}=1}^\infty \frac{e^{-m_{11}(x_{12} + \cdots x_{1h} + x_{21} + \cdots x_{\ell 1})}}{m_{11} ^{k_{11}}}   dx_{21} \cdots dx_{\ell1} dx_{12} \cdots dx_{1h}.
\end{eqnarray*}
From the definition of the $\Gamma$-function, we have
\begin{eqnarray*}
  &= & \frac{1}{ \Gamma(k_{12}) \cdots \Gamma(k_{1h})  \Gamma(k_{21}) \cdots \Gamma(k_{\ell1})} \int_0^{\infty} \frac{x_{1h}^{k_{1h}-1} e^{x_{1h}}}{e^{x_{1h}}-1} \int_0^\infty \cdots \int_0^\infty  \frac{x_{12}^{k_{12}-1}e^{(x_{12 }+ \cdots + x_{1h})}}{e^{(x_{12 }+ \cdots + x_{1h})}-1} \\
   && \int_0^{\infty}  \frac{x_{\ell 1}^{k_{\ell 1}-1}}{e^{x_{ \ell 1}}-1} \int_0^\infty  \cdots \int_0^\infty \frac{x_{21}^{k_{21}-1}}{e^{(x_{21 }+ \cdots + x_{\ell1})}-1}  \\
  & & \int_0^\infty \frac{1}{\Gamma(k_{11})}   \int_0^\infty \frac{x_{11}^{k_{11}-1}}{e^{x_{11} + \cdots +x_{1h} +x_{21} + \cdots x_{\ell1} }-1}   dx_{11} dx_{21} \cdots dx_{\ell1} dx_{12} \cdots dx_{1h}.
  \end{eqnarray*}
  This implies  \eqref{integralexpression}.
\end{proof}

\section{Bernoulli numbers}
As in Introduction, Schur type Bernoulli numbers are defined by $\eqref{defBernoulli}$ and  $\eqref{defBernoulliC}$ using $Li_{\mathbf{k}}^\lambda (\mathbf{z})$.
These numbers ${\mathbb B}_{{m_1}, \cdots ,{m_c}}^{\lambda, \mathbf{k}}$ and ${\mathbb C}_{{m_1}, \cdots ,{m_c}}^{\lambda, \mathbf{k}}$ are 
related to each other as in the following theorem.
\begin{Theorem}
 The following relations hold:
 \begin{eqnarray}
{\mathbb B}_{{m_1}, \cdots ,{m_c}}^{\lambda, \mathbf{k}}&=&
\sum_{ n_1 = 0}^{{m_1}} \cdots \sum_{ n_c = 0}^{{m_c}} \left(  \begin{matrix}
m_1 \\ n_1
\end{matrix} \right) 
\cdots
\left(  \begin{matrix}
m_c \\ n_c
\end{matrix} \right) 
{\mathbb C}_{{n_1}, \cdots {n_c}}^{\lambda, \mathbf{k}}, \label{bnmkl} \\
{\mathbb C}_{{m_1}, \cdots, {m_c}}^{\lambda, \mathbf{k}}&=&\sum_{ n_1 = 0}^{{m_1}} \cdots \sum_{ n_c = 0}^{{m_c}} 
(-1)^{\sum_i m_i - \sum_j n_j }
\left(  \begin{matrix}
m_1 \\ n_1
\end{matrix} \right) 
\cdots
\left(  \begin{matrix}
m_c \\ n_c
\end{matrix} \right) 
{\mathbb B}_{{n_1}, \cdots {n_c}}^{\lambda, \mathbf{k}}. \label{cnmkl}
 \end{eqnarray}
\end{Theorem}

\begin{proof}
Since
$$
\frac{{Li}_k^{\lambda}(1-e^{-z_1},  \cdots, 1-e^{-z_c})}{(1-e^{-z_1}) \cdots (1-e^{-z_c})}
=
e^{z_1} \cdots e^{z_c}\frac{{Li}_k^{\lambda}(1-e^{-z_1},  \cdots ,1-e^{-z_c})}{(e^{z_1}-1)  \cdots
(e^{z_c}-1)},
$$
\eqref{defBernoulli} and \eqref{defBernoulliC} give
$$
\sum_{{m_1}, \cdots ,{m_c}=0}^{\infty}{\mathbb B}_{{m_1}, \cdots, {m_c}}^{\lambda, \mathbf{k}}\frac{z_1^{{m_1}} \cdots z_c^{{m_c}}}{{m_1}!  \cdots {m_c}!}
=\sum_{s_1=0}^{\infty}\frac{z_1^{s_1}}{{s_1}!} \cdots \sum_{s_c=0}^{\infty}\frac{z_c^{s_c}}{{s_c}!} 
\sum_{{m_1}, \cdots ,{m_c}=0}^{\infty}{\mathbb C}_{{m_1}, \cdots ,{m_c}}^{\lambda, \mathbf{k}}\frac{z_1^{{m_1}} \cdots z_c^{{m_c}}}{{m_1}!  \cdots {m_c}!}.
$$
By changing the order of the last two summations, the right-hand side becomes 
$$
\sum_{s_1=0}^{\infty}\frac{z_1^{s_1}}{{s_1}!} \cdots \sum_{s_1=0}^{\infty}\frac{z_{m_{c-1}}^{s_{c-1}}}{{s_{c-1}}!}  
 \sum_{{m_1}, \cdots ,{m_c}=0}^{\infty} 
\sum_{n_c=0}^{{m_c}}{\mathbb C}_{{m_1}, \cdots ,{n_c}}^{\lambda, \mathbf{k}}\frac{z_1^{{m_1}} \cdots z_c^{{m_c}}}{{m_1}!  \cdots {m_c}!({m_c}-{n_c})!}.
$$
Changing the order of summation in the right-hand side repeatedly gives
$$
\sum_{{m_1}, \cdots , {m_c}=0}^{\infty}\sum_{ n_1 = 0}^{{m_1}} \cdots \sum_{ n_c = 0}^{{m_c}} \left(  \begin{matrix}
m_1 \\ n_1
\end{matrix} \right) 
\cdots
\left(  \begin{matrix}
m_c \\ n_c
\end{matrix} \right) 
{\mathbb C}_{{n_1}, \cdots , {n_c}}^{\lambda, \mathbf{k}}\frac{z_1^{{m_1}} \cdots z_c^{{m_c}}}{{m_1}!  \cdots {m_c}!}.
$$
By comparing the coefficients of the two sides, we have $(\ref{bnmkl})$. Equation $(\ref{cnmkl})$ can also be obtained similarly.
\end{proof}

Next we investigate a certain multiple zeta function which interpolates Schur type poly-Bernoulli numbers. This function, whose definition follows, can be regarded as a generalization of \eqref{AKMZ}.
\begin{Definition} \label{xidef}
Let
\begin{multline*}
\xi( \mathbf{k} ; s_1, \cdots ,s_c) =  \frac{1}{\Gamma(s_1) \cdots \Gamma(s_c)} \\
\int_{0}^\infty \cdots \int_{0}^\infty z_1^{s_1-1}  \cdots z_c^{s_c-1} \frac{Li_{\mathbf{k}}^\lambda(1-e^{-z_1}, \cdots , 1-e^{-z_c})}{(e^{z_1}-1) \cdots (e^{z_c}-1)} dz_1 \cdots dz_c
\end{multline*}
for $s_1, \cdots ,s_c \in \mathbb{C}$ with ${\rm Re}(s_i) > 0$.
\end{Definition}
The integral converge absolutely in ${\rm Re}(s_i) > 0$ because
$$
Li_{\mathbf{k}}^\lambda(1-e^{-z_1}, \cdots , 1-e^{-z_c}) \leq Li^\lambda_{\mathbf{k}}(1, \cdots , 1) = \zeta_\lambda(\mathbf{k}) .
$$
The special values  of $\xi^{\lambda}( \mathbf{k} ; s_1, \cdots ,s_c) $ at non-positive integers are related to $\mathbb{C}_{m_1, \cdots , m_c}^{\lambda, \mathbf{k}} $ as given in the following theorem.
\begin{Theorem} \label{xi}
The function $ \xi^{\lambda}( \mathbf{k} ; s_1,  \cdots , s_c) $ can be analytically continued to an entire function on the whole ${\mathbb C}^c$ plane. Furthermore, $\xi^{\lambda}( \mathbf{k} ; s_1, \cdots , s_c)$ at non-positive integers are given by
\begin{equation}
 \xi^{\lambda}( \mathbf{k} ; -m_1,  \cdots , -m_c) =  (-1)^{\sum_{j=1}^c m_j}\mathbb{C}_{m_1, \cdots , m_c}^{\lambda, \mathbf{k}} 
 \end{equation}
 for $m_i\geq 1$ being integers $(i=1, 2, \cdots, c)$.
 \end{Theorem}
We can prove Theorem 2 in various ways, but here we borrow the argument in \cite{KT}. 

\begin{proof}
We use the contour integral. Let $C'$ be the standard contour, which consists of a path from infinity to (sufficiently small) $\epsilon$, a circle $ C_\epsilon$ around the origin of radius $\epsilon$, and a path from $\epsilon$ to infinity. Let
\begin{multline*}
H^{\lambda}(\mathbf{k} ; s_1, \cdots , s_c) =\\
 \frac{1}{\Gamma(s_1) \cdots \Gamma(s_c)} \int_{C'} \cdots \int_{C'} z_1^{s_1-1}  \cdots z_c^{s_c-1} \frac{{Li}_{\mathbf{k}}^\lambda(1-e^{-z_1}, \cdots , 1-e^{-z_c})}{(e^{z_1}-1) \cdots (e^{z_c}-1)} dz_1 \cdots dz_c .\\
\end{multline*}
Then the right-hand side equals
\begin{eqnarray*}
& & \prod_{j=1}^c (e^{2 \pi i s_j}-1)  \int_{\epsilon}^{\infty} \cdots \int_{\epsilon}^{\infty} z_1^{s_1-1}  \cdots z_c^{s_c-1} \frac{Li_{\mathbf{k}}^\lambda(1-e^{-z_1}, \cdots , 1-e^{-z_c})}{(e^{z_1}-1) \cdots (e^{z_c}-1)} dz_1 \cdots dz_c  \\
&+& \prod_{j=2}^c (e^{2 \pi i s_j}-1) \int_{C_\epsilon}  \int_{\epsilon}^{\infty} \cdots \int_{\epsilon}^{\infty} z_1^{s_1-1}  \cdots z_c^{s_c-1} \frac{{Li}_{\mathbf{k}}^\lambda(1-e^{-z_1}, \cdots , 1-e^{-z_c})}{(e^{z_1}-1) \cdots (e^{z_c}-1)} dz_1 \cdots dz_c  \\
&+& \cdots \\
&+&  (e^{2 \pi i s_c}-1) \int_{C_\epsilon}  \cdots \int_{C_\epsilon}  \int_{\epsilon}^{\infty} z_1^{s_1-1}  \cdots z_c^{s_c-1} \frac{{Li}_{\mathbf{k}}^\lambda(1-e^{-z_1}, \cdots , 1-e^{-z_c})}{(e^{z_1}-1) \cdots (e^{z_c}-1)} dz_1 \cdots dz_c  \\
&+& \int_{C_{\epsilon}} \cdots \int_{C_{\epsilon}}  z_1^{s_1-1}  \cdots z_c^{s_c-1} \frac{{Li}_{\mathbf{k}}^\lambda(1-e^{-z_1}, \cdots , 1-e^{-z_c})}{(e^{z_1}-1) \cdots (e^{z_c}-1)} dz_1 \cdots dz_c  .
 \end{eqnarray*}
 $H^{\lambda}(\mathbf{k} ; s_1, \cdots , s_c) $ is entire because the integral has no singularity and is absolutely convergent for all $s_i \in \mathbb{C} (i=1, \cdots, c)$. Suppose ${\rm Re}(s_i) > 0$, then  $\displaystyle{\int_{C_\epsilon}}$ tends to 0 as $\epsilon \to 0$. Hence,
 \begin{align*}
 \xi^{\lambda}( \mathbf{k} ; s_1,  \cdots , s_c)=  \frac{1}{\prod_{j=1}^c (e^{2 \pi i s_j}-1) \Gamma(s_j) } H^{\lambda}(\mathbf{k} ; s_1, \cdots , s_c),
\end{align*}
which can be analytically continued to ${ \mathbb C}^n$ and is entire, and $\xi^{\lambda}$ is holomorphic for ${\rm Re}(s_1), \cdots , {\rm Re}(s_c) > 0$. 
For integers $m_1, \cdots, m_c\geq 1$,
 \begin{eqnarray*}
\xi^{\lambda}(\mathbf{k}; -m_1, \cdots , -m_c) &=& \prod_{j=1}^c \frac{(-1)^{m_j} {m_j} !}{(2 \pi i)^c}   H^{\lambda}(\mathbf{k} ; -m_1, \cdots , -m_c)  \\
&=& \frac{(-1)^{m_1 + \cdots + m_c} m_1! \cdots m_c !}{(2 \pi i)^c}   \sum_{a_1, \cdots ,a_c = 0}^{\infty} \frac{\mathbb{C}_{a_1,\cdots, a_c}^{\lambda, \mathbf{k}}}{a_1! \cdots a_c!} \\
&&  \int_{C_\epsilon} \cdots \int_{C_\epsilon} z_1^{-m_1 - 1 + a_1} \cdots z_c^{-m_c - 1 + a_c} dz_1 \cdots dz_c \\
&=&  \frac{(-1)^{m_1 + \cdots + m_c}  m_1! \cdots m_c !}{(2 \pi i)^c}\frac{\mathbb{C}_{m_1,\cdots, m_c}^{\lambda, \mathbf{k}}}{m_1! \cdots m_c!}  (2 \pi i)^c \\
&=&  (-1)^{\sum_{j=1}^c m_j}\mathbb{C}_{m_1, \cdots , m_c}^{\lambda, \mathbf{k}} .
 \end{eqnarray*} 
 \end{proof}

Note that we can also prove Theorems $\ref{xi}$ by the method in \cite{AKbook}, which divides the integral  in $\xi^{ \lambda}$ into two parts, $\int_0^{\infty} = \int_0^{1} + \int_1^{\infty} $.

\section{Hook-type Bernoulli numbers}
In this section, we restrict $\lambda$ to the hook type, that is, $\lambda=(h, 1^{\ell-1})$.
In what follows in this section, we use $z_h=z_{1h}$ and
$z_{\ell}=z_{\ell 1}$   for brevity, and we use $n$ and $m$ instead of $m_{1h}$ and
$m_{\ell 1}$ for the sake of clarity.
We write ${\bf k}=
\ytableausetup{boxsize=normal}  
\begin{ytableau}
k_{11} & k_{12} & \cdots & k_{1h}\\
k_{21} \\
\vdots \\
k_{\ell 1} \\
\end{ytableau}$ \quad as
${ \bf k}=(k_{\ell 1}, \cdots, k_{21} | k_{11}, k_{12}, \cdots, k_{1h})$, from bottom to top and left to right, which means $ k_{\ell 1} > \cdots > k_{21} > k_{11} \leq k_{12} \leq \cdots \leq k_{1h} $.
\begin{Lemma} \label{bibn}
For $\lambda = (h,1^{\ell-1})$ and $ \mathbf{k}=(k_{\ell 1}, \cdots, k_{21}| k_{11}, k_{12}, \cdots, k_{1h}) \in T(\lambda, \mathbb{N})$, 
\begin{eqnarray}
& \displaystyle{\frac{\partial}{\partial z_h}{Li}_\mathbf{k}^{\lambda}(z_h, z_\ell)
=\frac{1}{z_h}{Li}_
{(k_{\ell 1}, \cdots, k_{21}| k_{11}, k_{12}, \cdots, k_{1h}-1)}
^{\lambda}(z_h, z_\ell)},  \label{nlbibn1} \\
& \displaystyle{\frac{\partial}{\partial z_\ell}{Li}_\mathbf{k}^{\lambda}(z_h, z_\ell)
=\frac{1}{z_\ell}{Li}_
{(k_{\ell 1}-1, \cdots, k_{21}| k_{11}, k_{12}, \cdots, k_{1h})}
^{\lambda}(z_h, z_\ell)}, \label{nlbibn2} \\
& \displaystyle{\frac{\partial^2}{\partial z_h\partial z_\ell}{Li}_\mathbf{k}^{\lambda}(z_h, z_\ell)
 =\frac{1}{z_h z_\ell}{Li}_
 {(k_{\ell 1}-1, \cdots, k_{21}| k_{11}, k_{12}, \cdots, k_{1h}-1)}
^{\lambda}(z_h, z_\ell)}. \label{nlbibn3}
\end{eqnarray}
\end{Lemma}
\begin{proof}
For $( \ref{nlbibn1} )$,  a direct calculation leads to
 \begin{align*}
\frac{\partial}{\partial z_h}{Li}_\mathbf{k}^{\lambda}(z_h, z_\ell)
&=\sum_{\substack{m_{11}\leq \cdots \leq m_{1h} \\ m_{11} <\cdots <m_{\ell 1}}}
\frac{m_{1h}z_h^{m_{1h}-1}z_\ell^{m_{\ell 1}}}
{m_{\ell 1}^{k_{\ell 1}}\cdots m_{11}^{k_{11}} \cdots m_{1h}^{k_{1h}}}\\
&=\frac{1}{z_h}\sum_{\substack{m_{11}\leq \cdots m_{1h} \\ m_{11} <\cdots <m_{\ell 1}}}
\frac{z_h^{m_{1h}}z_\ell^{m_{\ell 1}}}{m_{\ell 1}^{k_{\ell 1}}\cdots m_{11}^{k_{11}} \cdots m_{1h}^{k_{1h}-1}}\\
&=\frac{1}{z_h}{Li}_{(k_{\ell 1}, \cdots, k_{21}| k_{11}, k_{12}, \cdots, k_{1h}-1)}^{\lambda}(z_h, z_\ell).
\end{align*}
We can obtain $( \ref{nlbibn2} )$ and $ (\ref{nlbibn3} )$ similarly.
\end{proof}

We can define a generalization of the $\eta$-functions \eqref{KTMZ} in the same way as in Definition $\ref{xidef}$.
\begin{Definition} 
Let
\begin{multline*}
\eta^{\lambda}( \mathbf{k} ; s_1, \cdots ,s_c) =  \frac{1}{\Gamma(s_1) \cdots \Gamma(s_c)} \\
\int_{0}^\infty \cdots \int_{0}^\infty z_1^{s_1-1}  \cdots z_c^{s_c-1} \frac{{Li}_{\mathbf{k}}^\lambda(1-e^{z_1}, \cdots , 1-e^{z_c})}{(1-e^{z_1}) \cdots (1-e^{z_c})} dz_1 \cdots dz_c
\end{multline*}
for $s_i \in \mathbb{C}$ with ${\rm Re}(s_i) > 1$.
\end{Definition}

For $\lambda = (h, 1^{\ell-1})$, the integral in $\eta^{\lambda} (\mathbf{k}; s_h, s_\ell)$ converges absolutely in ${\rm Re}(s_h) > 0$ and ${\rm Re}(s_\ell) > 1- \ell$ , as is seen from the following lemma which is extension of Lemma 2.2 in $\cite{KT}$. 

\begin{Lemma}
$(i)$ For $\mathbf{k} \in \mathbb{N}^{|\mathbf{k}|}$ and $ \lambda = (h, 1^{\ell-1}) $, $ {Li}_{\mathbf{k}}^\lambda(1-e^{z_h},  1-e^{z_\ell})$ is holomorphic for $z_h, z_\ell \in \mathbb{C} $ with $ | {\rm Im}(z_h) |$ and $| {\rm Im}(z_\ell) | < \pi $. \\
$(ii)$ For $\mathbf{k} \in \mathbb{N}^{|\mathbf{k}|} $ and $z_j \in \mathbb{R}_{>0}$, we have
\begin{equation} \label{Li0}
  {Li}_{\mathbf{k}}^\lambda(1-e^{z_h},  1-e^{z_\ell}) =  O(z_h z_\ell^\ell) \quad (z_h, z_\ell \to 0)
\end{equation} and
\begin{equation} \label{Liinfty}
  {Li}_{\mathbf{k}}^\lambda(1-e^{z_h},  1-e^{z_\ell})  = O(z_h^{k_{1h}} z_\ell^{k_{21}+ k_{31} + \cdots + k_{{\ell1}}}) \quad(z_h, z_\ell \to \infty).
\end{equation}
\end{Lemma}

\begin{proof}
(i)  By definition, ${Li}_{\mathbf{k}}^\lambda({z_h}, {z_\ell})$ is holomorphic for $ \mathbb{C} \backslash [ 1, \infty ) $. And ${Li}_{\mathbf{k}}^\lambda({z_h}, {z_\ell})$ is iterated integration: for $N = k_{12} + \cdots + k_{1h}$ and $M = k_{21} + \cdots + k_{\ell1} $, \\
\begin{align*}
 Li_{\mathbf{k}}^{\lambda} (z_h, z_\ell) = \underbrace{ \int_0^{z_h} \cdots \int_0^{z_h}}_{N \text{ times}} & \underbrace{\int_0^{z_\ell} \cdots \int_0^{z_\ell}}_{M \text{ times}}  \left( \frac{1}{z_h} \right)^{N-(h-1)} \left( \frac{1}{1-z_h} \right)^{h-1} \\
  & \left( \frac{1}{z_\ell} \right)^{M-(\ell-1)} \left( \frac{1}{1-z_\ell} \right)^{\ell-1} Li_{\mathbf{k}_{11}}^{(1)} (z_h z_\ell)  \underbrace{dz_{\ell} \cdots dz_{\ell}}_{M \text{ times}}  \underbrace{dz_{h} \cdots dz_{h}}_{N \text{ times}}    .
\end{align*}
Note that $1- e^{z_h}, 1-e^{z_\ell} \in [1,\infty)$  is equivalent to ${\rm Im}(z_h) $ and $ {\rm Im}(z_\ell) = (2j+1)\pi$ for any $j \in \mathbb{Z}$. \\
(ii) We have
 \begin{eqnarray*}
{Li}_{\mathbf{k}}^\lambda(z_h,  z_\ell) &=& \sum_{(m_{ij}) \in SSYT(\lambda)} \frac{{z_h}^{m_{1h}} (z_\ell )^{m_{\ell 1}}}{m_{11}^{k_{11}} \cdots m_{1h}^{k_{1h}} \cdots m_{\ell 1}^{k_{\ell 1}} } \\
&=& \frac{(z_h)^1 (z_\ell )^\ell}{1^{k_{11}} \cdots  1^{k_{1h}} 2^{k_{21}} 3^{k_{31}} \cdots \ell^{k_{\ell 1}} } + O(z_h^2 z_\ell^{\ell+1}) \quad (z_h, z_\ell \to 0).
 \end{eqnarray*}
 By this fomula, we have ${Li}_{\mathbf{k}}^\lambda(z_h,  z_\ell) = O(z_h^1 z_\ell^\ell ) $ as $ z_h, z_\ell \rightarrow 0$. Replacing $z_h$ and $z_\ell$ with $1-e^{z_h}$ and $1-e^{z_\ell}$ respectively, then 
  \begin{eqnarray*}
  {Li}_{\mathbf{k}}^\lambda(1-e^{z_h},  1-e^{z_\ell}) = O((1-e^{z_h})^1 (1-e^{z_\ell})^\ell) =  O(z_h z_\ell^\ell ) \quad (z_h, z_\ell \rightarrow 0).
  \end{eqnarray*} 
 Suppose $(k_{\ell 1}-1, \cdots, k_{21}| k_{11}, k_{12}, \cdots, k_{1h}-1) $ satisfies $\eqref{Liinfty}$. Using $\eqref{nlbibn3}$, we have
    \begin{eqnarray*}
    |  {Li}_{\mathbf{k}}^\lambda(1-e^{z_h},  1-e^{z_\ell}) | &=& \left | \int_0^{1-e^{z_h}} \int_0^{1-e^{z_\ell}} \frac{{Li}_{(k_{\ell 1}-1, \cdots, k_{21}| k_{11}, k_{12}, \cdots, k_{1h}-1)}^\lambda (u, v)}{uv} dvdu \right |.
      \end{eqnarray*}
      Repacing $u = 1-e^s, v= 1-e^t$, we have
      \begin{eqnarray*}
      &=&  \left | \int_0^{z_h} \int_0^{z_\ell} \frac{{Li}_{(k_{\ell 1}-1, \cdots, k_{21}| k_{11}, k_{12}, \cdots, k_{1h}-1)}^\lambda (1-e^s, 1-e^t)}{(1-e^s)(1-e^t)}(-e^s)(-e^t) dtds \right | \\
      &\leq& \int_0^{\epsilon} \int_0^{\epsilon} \left |  \frac{{Li}_{(k_{\ell 1}-1, \cdots, k_{21}| k_{11}, k_{12}, \cdots, k_{1h}-1)}^\lambda (1-e^s, 1-e^t)}{(e^s-1)(e^t-1)}e^s e^t \right | dtds  \\
      &+&  \int_0^{\epsilon} \int_\epsilon^{z_\ell} \left | \frac{{Li}_{(k_{\ell 1}-1, \cdots, k_{21}| k_{11}, k_{12}, \cdots, k_{1h}-1)}^\lambda (1-e^s, 1-e^t)}{(e^s-1)(e^t-1)}e^s e^t \right | dtds  \\
      &+&  \int_\epsilon^{z_h} \int_0^{\epsilon} \left | \frac{{Li}_{(k_{\ell 1}-1, \cdots, k_{21}| k_{11}, k_{12}, \cdots, k_{1h}-1)}^\lambda (1-e^s, 1-e^t)}{(e^s-1)(e^t-1)}e^s e^t \right | dtds \\
      &+&  \int_\epsilon^{z_h} \int_\epsilon^{z_\ell} \left | \frac{{Li}_{(k_{\ell 1}-1, \cdots, k_{21}| k_{11}, k_{12}, \cdots, k_{1h}-1)}^\lambda (1-e^s, 1-e^t)}{(e^s-1)(e^t-1)}e^s e^t \right | dtds 
      \end{eqnarray*}
      where $\epsilon$ is sufficiently small. First three terms in the right-hand side are $O(1)$ as $\epsilon \to 0$. From the assumption, the last term is $O(s^{k_{1h}-1} t^{k_{21}+ k_{31} + \cdots + k_{{\ell1}}-1})$ as $s,t \to \infty$. Therefore the last term is $O(z_h^{k_{1h}} z_\ell^{k_{21}+ k_{31} + \cdots + k_{{\ell1}}})$ as $z_h, z_\ell \to \infty $. We have $\eqref{Liinfty}$ by induction.
\end{proof}

It is natural questions to ask the analytically continuation of $\eta^\lambda (\mathbf{k}: s_1, \cdots, s_c)$ to an entire function on the whole $\mathbf{C}^c$ plane. For general $\lambda$, however, we have the difficulty  in the estimate of this type for $Li_{\mathbf{k}}^\lambda (1-e^{z_1}, \cdots, 1-e^{z_c})$ because we do not have the differentiation for general aspects in Lemma $\ref{bibn}$.

\begin{Theorem} \label{eta}
The function $ \eta^{\lambda}( \mathbf{k} ; s_1, s_2) $ of hook-type $\lambda$ can be analytically continued to an entire function on the whole ${\mathbb C}^2$ plane. 
In addition,
\begin{equation}
 \eta^{\lambda}(\mathbf{k} ; -m_1, -m_2) = \mathbb{B}_{m_1, m_2}^{\lambda, \mathbf{k}} .
 \end{equation}
\end{Theorem}

\begin{proof}
We apply the same method as for Theorem $\ref{xi}$. 
\end{proof}

We finally discuss some further properties in this situation such as a recurrence formula, a relation formula between
${\mathbb B_{n, m}^{\lambda, {\bf k}}}$ and ${\mathbb C_{n, m}^{\lambda, {\bf k}}}$ and a description in terms of the Stirling numbers.
First, using Lemma \ref{bibn}, we have the following proposition.
\begin{Proposition}
\begin{align*}
\sum_{n, m=0}^{\infty}
{\mathbb B}_{n, m}^{\lambda , \mathbf{k}^-}\frac{z_h^nz_\ell^m}{n!m!}
=&
\sum_{n, m=0}^{\infty}{\mathbb B}_{n, m}^{\lambda, \mathbf{k}}\frac{z_h^n z_\ell^m}{n! m!}
+
\sum_{v=1}^{\infty} \left(\sum_{m=0}^{v-1} \left(\sum_{n=0}^{\infty}{\mathbb B}_{n, m+1}^{\lambda, \mathbf{k}}\frac{z_h^n}{n!}\left(\begin{matrix}v \\ m\end{matrix}\right) \right ) \right)
\frac{z_\ell^{v}}{v !}\\
+&
\sum_{v'=1}^{\infty} \left( \sum_{n=0}^{v'-1} \left(\sum_{m=0}^{\infty}{\mathbb B}_{n+1, m}^{\lambda, \mathbf{k}}\frac{z_\ell^m}{m!}\left(\begin{matrix}v' \\ n\end{matrix}\right) \right)\right)
\frac{z_h^{v'}}{v' !}\\
+&
\sum_{v'=1}^{\infty}\sum_{n=0}^{v'-1}\sum_{v=1}^{\infty}\sum_{m=0}^{v-1}
{\mathbb B}_{n+1, m+1}^{\lambda, \mathbf{k}}
\left(\begin{matrix}v \\ m\end{matrix} \right)
\left(\begin{matrix}v' \\ n\end{matrix} \right)
\frac{z_\ell^{v}}{v!}\frac{z_h^{v'}}{v'!} 
,
\end{align*}
where ${ \bf k}^-$ means ${(k_{\ell 1}-1, \cdots, k_{21}| k_{11}, k_{12}, \cdots, k_{1h}-1)}$.
\end{Proposition}

\begin{proof}
Equation \eqref{defBernoulli} can be expressed as
$$
{Li}_{\bf k}^{\lambda}(1-e^{-z_h}, 1-e^{-z_\ell})
=(1-e^{-z_h})(1-e^{-z_\ell})\sum_{n, m=0}^{\infty}{\mathbb B}_{n, m}^{\lambda, \mathbf{k}}\frac{z_h^n z_\ell^m}{n! m!}.
$$

Differentiating with respect to $z_h$ and $z_\ell$ gives, by \eqref{nlbibn3}, 
\begin{align*}
&\frac{e^{-z_h}e^{-z_\ell}}{(1-e^{-z_h})(1-e^{-z_\ell})}
{Li}_{(k_{\ell 1}-1, \cdots, k_{21}| k_{11}, k_{12}, \cdots, k_{1h}-1)}^{\lambda}(1-e^{-z_h}, 1-e^{-z_\ell})\\
&=e^{-z_h}e^{-z_\ell}\sum_{n, m=0}^{\infty}{\mathbb B}_{n, m}^{\lambda, \mathbf{k}}\frac{z_h^n z_\ell^m}{n! m!}
+e^{-z_h}(1-e^{-z_\ell})\sum_{n=0}^{\infty}\sum_{m=1}^{\infty}{\mathbb B}_{n, m}^{\lambda, \mathbf{k}}\frac{z_h^n z_\ell^{m-1}}{n! (m-1)!}\\
&+(1-e^{-z_h})e^{-z_\ell}\sum_{n=1}^{\infty}\sum_{m=0}^{\infty}{\mathbb B}_{n, m}^{\lambda, \mathbf{k}}\frac{z_h^{n-1} z_\ell^m}{(n-1)!m!}\\
& +(1-e^{-z_h})(1-e^{-z_\ell})\sum_{n, m=1}^{\infty}{\mathbb B}_{n, m}^{\lambda, \mathbf{k}}\frac{z_h^{n-1} z_\ell^{m-1}}{(n-1)! (m-1)!}
.
\end{align*}
Multiplying both sides by $e^{z_h}e^{z_\ell}$, we have
\begin{align} \label{R14}
&\sum_{n,m=0}^{\infty}
{\mathbb B}_{n, m}^{\lambda, (k_{\ell 1}-1, \cdots, k_{21}| k_{11}, k_{12}, \cdots, k_{1h}-1)}\frac{z_h^nz_\ell^m}{n!m!} \notag \\
&=\sum_{n, m=0}^{\infty}{\mathbb B}_{n, m}^{\lambda, \mathbf{k}}\frac{z_h^n z_\ell^m}{n! m!}
+(e^{z_\ell}-1)\sum_{n=0}^{\infty}\sum_{m=1}^{\infty}{\mathbb B}_{n, m}^{\lambda, \mathbf{k}}\frac{z_h^n z_\ell^{m-1}}{n! (m-1)!}\notag \\
&+(e^{z_h}-1)\sum_{n=1}^{\infty}\sum_{m=0}^{\infty}{\mathbb B}_{n, m}^{\lambda, \mathbf{k}}\frac{z_h^{n-1} z_\ell^m}{(n-1)!m!} 
+(e^{z_h}-1)(e^{z_\ell}-1)\sum_{n, m=1}^{\infty}{\mathbb B}_{n, m}^{\lambda, \mathbf{k}}\frac{z_h^{n-1} z_\ell^{m-1}}{(n-1)! (m-1)!}.
\end{align}
We write the right-hand side of \eqref{R14} as
$ (R1) + (R2) + (R3) + (R4)  .$\\
\\
The second term, $(R2)$, is
$$
(R2) =
\sum_{v=1}^{\infty}\frac{z_\ell^{v}}{v !}\sum_{m=1}^{\infty}
\left( \sum_{n=0}^{\infty}{\mathbb B}_{n, m}^{\lambda, \mathbf{k}}\frac{z_h^n}{n!} \right) \frac{z_\ell^{m-1}}{(m-1)!}.
$$
Setting $A_m=(\sum_{n=0}^{\infty}{\mathbb B}_{n, m}^{\lambda, { \bf k}}\frac{z_h^n}{n!})$ gives
\begin{align*}
(R2) &= \left(\sum_{v=0}^{\infty}\frac{z_\ell^v}{v !}-1\right)\sum_{m=0}^{\infty}A_{m+1}\frac{z_\ell^m}{m!}
=\sum_{v=0}^{\infty} \left( \sum_{m=0}^{v}A_{m+1}\frac{1}{m!(v-m)!} \right)z_\ell^{v}
-\sum_{m=0}^{\infty}A_{m+1}\frac{z_\ell^{m}}{m !}\\
&=\sum_{v=0}^{\infty} \left(\sum_{m=0}^{v}A_{m+1}\left(\begin{matrix}v \\ m\end{matrix}\right) \right)
\frac{z_\ell^{v}}{v !}
-\sum_{v=0}^\infty A_{v+1}\frac{z_\ell^{v}}{v !}
=\sum_{v=1}^{\infty} \left( \sum_{m=0}^{v-1}A_{m+1} \left (\begin{matrix}v \\ m\end{matrix} \right)\right)
\frac{z_\ell^{v}}{v !}.
\end{align*}
From the definition of $A_m$, we have
$$
(R2) =\sum_{v=1}^{\infty} \left(\sum_{m=0}^{v-1} \left(\sum_{n=0}^{\infty}{\mathbb B}_{n, m+1}^{\lambda, \mathbf{k}}\frac{z_h^n}{n!}\left(\begin{matrix}v \\ m\end{matrix}\right) \right ) \right)
\frac{z_\ell^{v}}{v !}.
$$
The third term, $(R3)$, can be calculated similarly:
$$
(R3) = \sum_{v'=1}^{\infty} \left( \sum_{n=0}^{v'-1} \left(\sum_{m=0}^{\infty}{\mathbb B}_{n+1, m}^{\lambda, \mathbf{k}}\frac{z_\ell^m}{m!}\left(\begin{matrix}v' \\ n\end{matrix}\right) \right)\right)
\frac{z_h^{v'}}{v' !}.
$$
Applying a calculation similar to that for (R2) twice to the fourth term, $(R4)$, gives
\begin{align*}
(R4) = &(e^{z_h}-1)(e^{z_\ell}-1)\sum_{n, m=1}^{\infty}
{\mathbb B}_{n, m}^{\lambda, \mathbf{k}}\frac{z_h^{n-1}z_\ell^{m-1}}{(n-1)!(m-1)!}\\
&=\sum_{v'=1}^{\infty}\frac{z_h^{v'}}{v'!}\sum_{v=1}^{\infty}\frac{z_\ell^{v}}{v !}
\sum_{n,m=1}^{\infty}{\mathbb B}_{n, m}^{\lambda, \mathbf{k}}\frac{z_h^{n-1}z_\ell^{m-1}}{(n-1)!(m-1)!}
\\
&=
\left(\sum_{v'=0}^{\infty}\frac{z_h^{v'}}{v'!}-1\right) \left(\sum_{v=0}^{\infty}\frac{z_\ell^{v}}{v!}-1 \right)
\sum_{m=0}^{\infty}\left (\sum_{n=0}^{\infty}{\mathbb B}_{n+1, m+1}^{\lambda, \mathbf{k}}\frac{z_h^{n}}{n!} \right)\frac{z_\ell^{m}}{m!}\\
&=
\left(\sum_{v'=0}^{\infty}\frac{z_h^{v'}}{v'!}-1\right) 
\left(
\sum_{n=0}^{\infty}\sum_{v=1}^{\infty}\sum_{m=0}^{v-1}
{\mathbb B}_{n+1, m+1}^{\lambda, \mathbf{k}}
\left(\begin{matrix}v \\ m\end{matrix} \right)
\frac{z_\ell^{v}}{v!}\frac{z_h^{n}}{n!} \right)\\
&=
\sum_{v'=1}^{\infty}\sum_{n=0}^{v'-1}\sum_{v=1}^{\infty}\sum_{m=0}^{v-1}
{\mathbb B}_{n+1, m+1}^{\lambda, \mathbf{k}}
\left(\begin{matrix}v \\ m\end{matrix} \right)
\left(\begin{matrix}v' \\ n\end{matrix} \right)
\frac{z_\ell^{v}}{v!}\frac{z_h^{v'}}{v'!} .
\end{align*}
Combining $(R1), (R2), (R3)$, and $(R4)$ gives the proposition.
\end{proof}

The next proposition is about the relation between Schur type Bernoulli numbers 
${\mathbb B}_{{n}, m}^{\lambda, \mathbf{k}}$ and ${\mathbb C}_{{n}, m}^{\lambda, \mathbf{k}}$.

\begin{Proposition}
For $\lambda = (h, 1^{\ell -1})$ and $k_{ij}\not=1$ for $(i, j)\in C( \lambda)$, the following holds:
$$
{\mathbb B}_{{n}, m}^{\lambda, \mathbf{k}}={\mathbb C}_{{n}, {m}}^{\lambda, \mathbf{k}}
+{\mathbb C}_{{n}-1, {m}}^{\lambda, 
(k_{\ell 1}, \cdots| k_{11}, \cdots, k_{1h}-1)}
+{\mathbb C}_{{n}, {m}-1}^{\lambda,
(k_{\ell 1}-1, \cdots| k_{11}, \cdots, k_{1h})}
+{\mathbb C}_{{n}-1, {m}-1}^{\lambda,
(k_{\ell 1}-1, \cdots| k_{11}, \cdots, k_{1h}-1)}.
$$
\end{Proposition}

\begin{proof}
From the definition of  Schur type Bernoulli numbers, 
\begin{align}
\sum_{n, m=0}^{\infty}{\mathbb B}_{n, m}^{\lambda, {\bf k}}
\frac{z_h^{n} z_\ell^{m}}{n! m!}
&=
\frac{{Li}_{\bf k}^{\lambda}(1-e^{-z_h})(1-e^{-z_\ell})}{(1-e^{-z_h})(1-e^{-z_\ell})}=e^{z_h}e^{z_\ell}\frac{{Li}_{\bf k}^{\lambda}(1-e^{-z_h}, 1-e^{-z_\ell})}{(e^{z_h}-1)(e^{z_\ell}-1)}\notag\\
&
=\frac{1}{(e^{z_h}-1)(e^{z_\ell}-1)}
\{(e^{z_h}-1)(e^{z_\ell}-1){Li}_{\bf k}^{\lambda}(1-e^{-z_h}, 1-e^{-z_\ell})\notag\\
&
+(e^{z_h}-1){Li}_{\bf k}^{\lambda}(1-e^{-z_h}, 1-e^{-z_\ell})
\notag\\
&
+(e^{z_\ell}-1){Li}_{\bf k}^{\lambda}(1-e^{-z_h}, 1-e^{-z_\ell})\notag\\
&
+{Li}_{\bf k}^{\lambda}(1-e^{-z_h}, 1-e^{-z_\ell})
\}\notag\\
&= {Li}_{\bf k}^{\lambda}(1-e^{-z_h}, 1-e^{-z_\ell})
+\frac{{Li}_{\bf k}^{\lambda}(1-e^{-z_h}, 1-e^{-z_\ell})}{e^{z_\ell}-1}\notag\\
&
+\frac{{Li}_{\bf k}^{\lambda}(1-e^{-z_h}, 1-e^{-z_\ell})}{e^{z_h}-1}
+\frac{{Li}_{\bf k}^{\lambda}(1-e^{-z_h}, 1-e^{-z_\ell})}{(e^{z_h}-1)(e^{z_\ell}-1)}.
\label{LikBnm}
\end{align}
Using \eqref{nlbibn3} with the definition of Bernoulli numbers \eqref{defBernoulliC}, 
\begin{align*}
{Li}_{\bf k}^{\lambda}&(1-e^{-z_h}, 1-e^{-z_\ell})\\
&=\int_{0}^{z_h}\int_0^{z_\ell}
\frac{{Li}_{k_{\ell 1}-1, \cdots| k_{11}, \cdots k_{1h}-1}^\lambda(1-e^{-z_h}, 1-e^{-z_\ell})}
{(e^{z_h}-1)(e^{z_\ell}-1)}
dz_\ell dz_h\\
&=\int_{0}^{z_h}\int_0^{z_\ell}
\sum_{n, m=0}^{\infty}{\mathbb C}_{n, m}^{(k_{\ell 1}-1, \cdots| k_{11}, \cdots, k_{1h}-1)}\frac{z_h^{n}}{{n}!}\frac{{z_\ell}^{m}}{m!}
dz_\ell dz_h\\
&=\sum_{n, m=0}^{\infty}{\mathbb C}_{n, m}^{(k_{\ell 1}-1, \cdots| k_{11}, \cdots, k_{1h}-1)}\frac{z_h^{n+1}}{(n+1)!}\frac{{z_\ell}^{m+1}}{(m+1)!}
\\
&=\sum_{n, m=1}^{\infty}{\mathbb C}_{n-1, m-1}^{(k_{\ell 1}-1, \cdots| k_{11}, \cdots, k_{1h}-1)}\frac{z_h^{n}}{{n}!}\frac{{z_\ell}^{m}}{m!}.
\end{align*}
After a similar calculation using \eqref{nlbibn1} and \eqref{nlbibn2}, comparing the coefficients in the two sides,
for $n, m\geq 1$, gives 
$$
{\mathbb B}_{n, m}^{\lambda, \mathbf{k}}=
{\mathbb C}_{n-1, m-1}^{\lambda(k_{\ell 1}-1, \cdots  | {k_{11}}, \cdots, k_{1h}-1)}
+{\mathbb C}_{n-1, m}^{\lambda(k_{\ell 1}, \cdots| k_{11}, \cdots, k_{1h}-1)}
+{\mathbb C}_{n, m-1}^{\lambda(k_{\ell 1}-1, \cdots| k_{11}, \cdots, k_{1h})}
+{\mathbb C}_{n, m}^{\lambda, \mathbf{k}}.
$$
\end{proof}

Lastly, we discuss a description in terms of the Stirling numbers.
For any integers $n$ and $m$,  Stirling number $\left \{
\begin{matrix} n\\m
\end{matrix}
\right \}$
is defined as follows:
$$
\left \{
\begin{matrix} n+1\\m
\end{matrix}
\right \}
=
\left \{
\begin{matrix} n\\m-1
\end{matrix}
\right \}
+
m \left \{
\begin{matrix} n\\m
\end{matrix}
\right \},
$$
where we set $\left \{
\begin{matrix} 0 \\0
\end{matrix}
\right \} = 1$, 
$\left \{
\begin{matrix} n \\0
\end{matrix}
\right \} = \left \{
\begin{matrix} 0 \\m
\end{matrix}
\right \} = 0 \; ((n, m) \neq (0, 0)) $.
The hook-type Schur  Bernoulli numbers can be written in terms of the Stirling numbers.
\begin{Proposition} 
For $h, \ell$ $\in \mathbb{N}$, $\lambda = (h, 1^{\ell-1})$, and $ \mathbf{k} = (k_{ij}) \in \mathbb{Z}^{|\lambda|} $, the following holds:
 $$
{\mathbb B}_{n, m}^{\lambda, \mathbf{k}}
 =
 \sum_{\substack{m_{11}\leq \cdots m_{1m} \leq m+1 \\  m_{11}<\cdots < m_{\ell 1}\leq n+1 }}
  \frac{(-1)^{m_{1h}+m_{\ell 1}+n+m}(m_{1h}-1)!(m_{\ell 1}-1)!}{m_{11}^{k_{11}}\cdots m_{1h}^{k_{1h}}\cdots m_{\ell 1}^{k_{\ell 1}}}
\left \{
\begin{matrix} n\\
m_{1m}-1
\end{matrix}
\right \}
\left \{
\begin{matrix}m\\
m_{\ell 1}-1
\end{matrix}
\right \}
.
$$
\end{Proposition}
\begin{proof}
From the definition of the Bernoulli numbers,
\begin{align*}
&\sum_{n, m=0}^{\infty}{\mathbb B}_{n, m}^{\lambda, \mathbf{k}}\frac{z_n^{n} z_\ell^{m}}{n! m!}
=\frac{{Li}_{\bf k}(1-e^{-z_h}, 1-e^{-z_\ell})}{(1-e^{-z_h})(1-e^{-z_\ell})}\\
&=\sum_{\substack{m_{11}\leq \cdots  \leq m_{1h}\\ m_{11}<\cdots < m_{\ell 1}}}
\frac{(1-e^{-z_h})^{m_{1h}-1}(1-e^{-z_\ell})^{m_{\ell 1}-1}}{m_{\ell 1}^{k_{\ell 1}}\cdots m_{11}^{k_{11}}\cdots m_{1h}^{k_{1h}}}\\
&=\sum \frac{(-1)^{(m_{1h}-1)+(m_{\ell 1}-1)}(m_{1h}-1)!(m_{\ell 1}-1)!}{m_{\ell 1}^{k_{\ell 1}}\cdots m_{11}^{k_11}\cdots m_{1h}^{k_{1h}}}
\frac{(e^{-z_h}-1)^{m_{1h}-1}}{(m_{1h}-1)!}\frac{(e^{-z_\ell}-1)^{m_{\ell 1}-1}}{(m_{\ell 1}-1)!}\\
&=\sum \frac{(-1)^{m_{1h}+m_{\ell 1}}(m_{1h}-1)!(m_{\ell 1}-1)!}{m_{\ell 1}^{k_{\ell 1}}\cdots m_{11}^{k_{11}}\cdots m_{1h}^{k_{1h}}}\cdot\\
&\hspace{1cm}
\sum_{p=m_{1h}-1}^{\infty}
\left \{
\begin{matrix} p\\
m_{1h}-1
\end{matrix}
\right \}
\frac{(-z_h)^p}{p!}\sum_{q=m_{\ell 1}-1}^{\infty} 
\left \{
\begin{matrix} q\\
m_{\ell 1}-1
\end{matrix}
\right \}
\frac{(-z_\ell)^q}{q!}.
\end{align*}
For the last equality, we used Proposition 2.6 (7) in $\cite{AK2}$:
$$
\frac{(e^t-1)^m}{m!}=\sum_{n=m}^{\infty}
\left \{
\begin{matrix} n\\
m
\end{matrix}
\right \}
\frac{t^n}{n!}.
$$
If we set $p-m_{1h}+1=\alpha$ and $q-m_{\ell 1}+1=\beta$,
 then the last equality is equivalent to
\begin{eqnarray*}
\sum \frac{(-1)^{m_{1h}+m_{\ell 1}}(m_{1h}-1)!(m_{\ell 1}-1)!}{m_{\ell 1}^{k_{\ell 1}}\cdots m_{11}^{k_{11}}\cdots m_{1h}^{k_{1h}}}
\sum_{\alpha=0, \beta=0}^{\infty}
\left \{
\begin{matrix} \alpha+m_{1h}-1\\
m_{1h}-1
\end{matrix}
\right \}
\left \{
\begin{matrix} \beta+m_{\ell 1}-1\\
m_{\ell 1}-1
\end{matrix}
\right \} \\
\frac{(-z_h)^{\alpha+m_{1h}-1}}{(\alpha+m_{1h}-1)!}
\frac{(-z_\ell)^{\beta+m_{\ell 1}-1}}{(\beta+m_{\ell 1}-1)!}.
\end{eqnarray*}
Replacing $\alpha+m_{1h}$ by $n$ and $\beta+m_{\ell 1}$ by $m$, we have
\begin{align*}
& =\sum_{n, m=1}^{\infty}
 \sum_{\substack{m_{11}\leq \cdots  < m_{1h} \leq m \\ m_{11}<\cdots < m_{\ell 1} \leq n}}
  \frac{(-1)^{m_{1h}+m_{\ell 1}}(m_{1h}-1)!(m_{\ell 1}-1)!}{m_{\ell 1}^{k_{\ell 1}}\cdots m_{11}^{k_11}\cdots m_{1h}^{k_{1h}}}
  \left \{
\begin{matrix} n-1\\
m_{1h}-1
\end{matrix}
\right \}
\left \{
\begin{matrix}m-1\\
m_{\ell 1}-1
\end{matrix}
\right \}\\
& \hspace{10cm}\cdot 
\frac{(-z_h)^{n-1}(-z_\ell)^{m-1}}{(n-1)!(m-1)!}\\
& =\sum_{n, m=0}^{\infty}
 \sum_{\substack{m_{11}\leq \cdots \leq m_{1h}\leq m+1 \\ m_{11}<\cdots < m_{\ell 1}\leq n+1}}
  \frac{(-1)^{m_{1h}+m_{\ell 1}}(m_{1h}-1)!(m_{\ell 1}-1)!}{m_{\ell 1}^{k_{\ell 1}}\cdots m_{11}^{k_{11}}\cdots m_{1h}^{k_{1h}}}
\left \{
\begin{matrix} n\\
m_{1h}-1
\end{matrix}
\right \}
\left \{
\begin{matrix}m\\
m_{\ell 1}-1
\end{matrix}
\right \}\\
&\hspace{10cm}
\cdot \frac{(-z_n)^{n}(-z_\ell)^{m}}{n!m!}.
 \end{align*}
 Comparing the coefficients of the two sides, we have the proposition.
 \end{proof}

\section*{Acknowledgement}
 We would like to express our appreciation to  
 Prof. Yasuo Ohno for his guidance and helpful comments in $C$-type Bernoulli numbers  
 and Prof. Hirofumi Tsumura for pointing out a mistake in convergence and his meaningful suggestion for our work.
 This work was supported in part by JSPS KAKENHI Grant Number JP18K03223.

\bigskip

\noindent
\textsc{Naoki Nakamura}\\
 Department of Information and Communication Science, Faculty of Science, \\
 Sophia University, Tokyo, Japan \\

\medskip

\noindent
\textsc{Maki Nakasuji}\\
 Department of Information and Communication Science, Faculty of Science, \\
 Sophia University, Tokyo, Japan \\
 \texttt{nakasuji@sophia.ac.jp}

\end{document}